\documentclass{amsart}

\usepackage[english]{babel}

\usepackage{amssymb}
\usepackage{amsmath}
\usepackage{amsthm}
\usepackage{enumerate}
\usepackage{xcolor}

\newcommand{\enm}[1]{\ensuremath{#1}}

\newcommand{\CC}{\enm{\mathbb{C}}}
\newcommand{\NN}{\enm{\mathbb{N}}}
\newcommand{\PP}{\enm{\mathbb{P}}}

\newcommand{\Aa}{\enm{\mathcal{A}}}
\newcommand{\Ii}{\enm{\mathcal{I}}}
\newcommand{\Jj}{\enm{\mathcal{J}}}
\newcommand{\Oo}{\enm{\mathcal{O}}}
\newcommand{\Ss}{\enm{\mathcal{S}}}

\newcommand{\red}{\operatorname{red}}

\newcommand{\refthm}[1]{{Theorem \ref{#1}}}
\newcommand{\refeqn}[1]{{(\ref{#1})}}
\newcommand{\refrem}[1]{{Remark \ref{#1}}}

\newcommand{\refex}[1]{{Example \ref{#1}}}

\newcommand{\refcor}[1]{{Corollary \ref{#1}}}

\newcommand{\refprop}[1]{{Proposition \ref{#1}}}
\newcommand{\reflem}[1]{{Lemma \ref{#1}}}

\newcommand{\spa}[1]{\langle{#1}\rangle}

\newcommand{\Res}{\operatorname{Res}}
\newcommand{\sym}{\operatorname{Sym}}

\newtheorem{definition}{Definition}[section]
\newtheorem{proposition}[definition]{Proposition}
\newtheorem{theorem}[definition]{Theorem}
\newtheorem{lemma}[definition]{Lemma}
\newtheorem{corollary}[definition]{Corollary}
\newtheorem{example}[definition]{Example}
\newtheorem{remark}[definition]{Remark}

\DeclareMathOperator{\SF}{{\rm SF}}

\begin{document}

\title[Bounds on the rank]
{Bounds on the {tensor rank}}
\date{}

%\subjclass[2000]{}

\author[E. Ballico, A. Bernardi, L. Chiantini, E. Guardo]{Edoardo Ballico, Alessandra Bernardi, Luca Chiantini, Elena Guardo}
\address[E. Ballico, A. Bernardi]{Dipartimento di Matematica, Universit\`a di Trento}
\email{edoardo.ballico@unitn.it, alessandra.bernardi@unitn.it}
\address[L. Chiantini]{Dipartimento di Ingegneria dell'Informazione e Scienze Matematiche
\\ Universit\`a di Siena} \email{luca.chiantini@unisi.it}
\address[E. Guardo]{Dipartimento di Matematica e Informatica, Universit\`a di Catania} \email{guardo@dmi.unict.it}

\maketitle

\begin{abstract} We give a sufficient criterion for a lower bound of the cactus rank of a tensor. 
Then we refine that criterion in order to be able to give an explicit sufficient condition for a non-redundant 
decomposition of a tensor to be minimal and unique.
\end{abstract}

\section{Introduction}

{The study of  minimal decompositions of a given tensor $T$ as a linear combination of rank one tensors is a hot}
topic in many areas, ranging {from} pure algebraic geometry to applications {to} signal processing, big data analysis, quantum information...
\medskip

{Vectors} $v_{j,i}\in \mathbb{C}^{n_j+1}$, $j=1, \ldots, k$, $i=1, \ldots, r$, such that 
\begin{equation}\label{decomposition}
T=\sum_{i=1}^r \lambda_i v_{1,i}\otimes \cdots \otimes v_{k,i}
\end{equation}
for some $\lambda_i\in \mathbb{C}\setminus \{0\}$
{ determine a {\it decomposition} of $T$. We will say that the decomposition is 
{\it non-redundant} (cf. Definition \ref{Definition:NonRedundant}) if }  we
cannot extract any proper subset {of} $\{v_{1,i}\otimes
\cdots \otimes v_{k,i}\}_{i=1, \ldots , r}$ { which generates $T$}. 

{Since we will use geometric arguments through the paper, we use a geometric notation. Thus
 we identify (up to scalar multiplication) a tensor $T$
with a point in the projective space $\PP(\mathbb{C}^{n_1+1}\otimes\dots\otimes\mathbb{C}^{n_k+1})$
and a decomposition of $T$ with a finite subset $S$ of the Segre embedding of  the} abstract product 
$\PP(\mathbb{C}^{n_1+1})\times\dots\times\PP(\mathbb{C}^{n_k+1})=$ $\PP^{n_1}\times\dots\times\PP^{n_k}$,
{such that $T$ belongs to the linear span of $S$}.
\medskip

Cleary, having a non-redundant decomposition of a given tensor $T$
does not imply that such a decomposition is {minimal, i.e. it has a minimal} 
number of addenda ({so that $r$ is} the rank
$\mathrm{rk}(T)$ of $T$, cf. Definition \ref{Definition:Ranks}).

In this paper we give a criterion that certifies if a
non-redundant decomposition of a general tensor $T$ { is} also minimal, {thus that it computes the rank of $T$ (cf. \refthm{add7}}). 
Moreover, under certain conditions, we can also show
that {a decomposition} is the unique minimal decomposition of {$T$} (cf. Theorem \ref{ee4}).

These two facts rely on {our main result} Theorem \ref{add7},
where we give a criterion to find a lower bound for the cactus rank (cf. Definition \ref{Definition:Ranks}).
The idea is {geometrically} quite simple: Assume one has a non-redundant decomposition {$S$} of a tensor 
$T\in\PP(\mathbb{C}^{n_1+1}\otimes \cdots \otimes \mathbb{C}^{n_k+1}) $, 
then one can flatten the product $\PP^{n_1}\times\dots\times\PP^{n_k}$ 
in a partition of two factors and study {the geometry of the two projections of $S$,
to get the result.} 

{One} can easily compare our {result} with the Kruskal's {result on the identifiability
of tensors (\cite{Kru}), which implies a criterion for the minimality of a decomposition. It turns out } 
that  our criterion {is geometrically simpler, and it} applies in a wider range of {numerical} cases.

{Theorem \ref{add7} has the} consequence that if a non-redundant decomposition projects onto two
linearly independents subset in the flattening, then it is also a minimal one (cf. Example \ref{Strassen}).
This {can} be  considered as a step towards the celebrated
Strassen's Conjecture on the rank of a block tensor (cf. \cite{Stra}).

{We show also that Theorem \ref{add7} provides new evidences towards the
Comon's Conjecture, i.e. the equality between the rank and the symmetric rank of a symmetric tensor $T$. 
In a wide numerical range, if we know the existence of a symmetric decomposition of $T$ with
sufficiently general geometric properties, then the conjecture holds for $T$
(cf. \refcor{comon}). For instance, the Corollary applies for general tensors in $\PP(\sym^6(\CC^3))$ and
$\PP(\sym^8(\CC^3))$, i.e. for general forms of degree $6$ and $8$ in three variables.}

Another {consequence} of Theorem \ref{add7} is {described} in Theorem \ref{add11}. 
{Given a minimal decomposition with $r$ addenda,
then for} any integer $r'$ such that $r\leq r' <-1+\Pi (n_i+1)$, it is always possible to find a non-redundant 
{decomposition with $r'$ addenda. We notice that this happens to be false for symmetric decompositions
of a} symmetric tensor{, even in the case  of} $S^d(\mathbb{C}^2)$ with $d\geq 4$ 
(cf. Sylvester Theorem in e.g. \cite{cs,bgi}).
\medskip

In Section \ref{Section:Identifiability} we focus on the identifiability of a minimal decomposition, 
meaning the uniqueness of a decomposition of a tensor $T$ with exactly $\mathrm{rk}(T)$ addenda. 
The main result of this section is Theorem \ref{ee4} (that is again a consequence of Theorem \ref{add7}), 
where we {point out}  a sufficient condition for a non-redundant decomposition to be minimal and unique. 
{Again we} compare our result with Kruskal's bound (\cite{Kru}). Since Kruskal's bound is sharp (\cite{sb}, \cite{Der}) 
and since our {geometric assumptions are weaker, and then easier to verify, than Kruskal's ones, 
we cannot hope to} produce
applications outside  Kruskal's numerical range. There are few cases in which {our}
the numerical range of application matches with Kruskal's range.
One of them e.g. is given by tensors of type $3\times 2 \times 2 \times 2$.

\section{Notation and preliminaries}

For a subscheme $Z\subset\PP^m$, we indicate with $\spa{Z}$
the linear span of $Z$ and with $\deg(Z)$ its length (when it is finite).
If $Z$ is finite and reduced, we indicate with $\sharp Z$ the cardinality of $Z$. %($=\deg(Z)$).
\medskip

For any product of projective spaces $\PP^{n_1}\times\dots\times\PP^{n_k}$ call $\nu$ the Segre map
$$\nu :\PP^{n_1}\times\dots\times\PP^{n_k}\to \PP^M,\qquad M=-1+\Pi (n_i+1).$$

In order to have a more compact notation we will always write $$Y:=\PP^{n_1}\times\dots\times\PP^{n_k}$$ for the abstract product,
and $$X:=\nu (Y)\subset \mathbb{P}^M$$ for the Segre variety.

\medskip

For any $i=1,\dots,k$ call $\pi_i$ the projection of  $\PP^{n_1}\times\dots\times\PP^{n_k}$ to the $i$-th factor.

We can generalize this notation by setting, for any collection of sub-indices  $\mathbf{u}=\{u_1,\dots,u_i\}\subset \{1,\dots,k\}$,
\begin{center}$\pi_{\mathbf{u}} =$ projection to the product of the factors $u_1,\dots, u_i$.
\end{center}
In particular $\pi_{\{1,\dots,i\}}$ is the projection to the product of the first $i$ factors.
\smallskip

For any subset $\mathbf{u}=\{u_1,\dots,u_i\}\subset \{1,\dots,k\}$, we will denote with $\Oo(\mathbf{u})$ the line
bundle on $\PP^{n_1}\times\dots\times\PP^{n_k}$ pull back of the hyperplane bundle on $\nu(\pi_{\mathbf{u}}(Y))$.

We will write simply $\Oo(1)$ when $\mathbf{u}=\{1,\dots,k\}$, i.e. $\Oo(1)$ is the pull back of
the hyperplane bundle in the Segre embedding of $\PP^{n_1}\times\dots\times\PP^{n_k}$.

For any subset $Z$ of $\PP^{n_1}\times\dots\times\PP^{n_k}$ we will write consequently
$$\Ii _Z(\mathbf{u})=\Ii_Z\otimes\Oo(\mathbf{u}),\quad \Oo _Z(\mathbf{u})=\Oo_Y\otimes\Oo(\mathbf{u})$$
and write $\Ii_Z(1),\Oo_Z(1)$ when $\mathbf{u}=\{1,\dots,k\}.$

Notice that the dimension $h^0(\Ii_Z(\mathbf{u}))$ corresponds then
to the co-dimension of the linear span of $\nu(\pi_{\mathbf{u}}(Z))$.
Obviously, if $Z$ is zero-dimensional and $h^1(\Ii _Z(\mathbf{u}))=0$, then also
$h^1(\Ii _Z(\mathbf{u}'))=0$ for all $\mathbf{u}'\supseteq \mathbf{u}$.
\medskip

We will say that a finite subset $S\subset \PP^{n_1}\times\dots\times\PP^{n_k}$ has {\em different coordinates}
if for all $i=1,\dots,k$ the projection $\pi_i$ to the $i$-th factor is an embedding of $S$
into $\PP^{n_i}$.
\medskip

We will need the process of {\it residuation} with respect to a divisor.
Let $X$ be a variety %(in our setting $X$ will be a product of projective spaces).
For any zero-dimensional scheme $Z\subset X$ and for any effective divisor
$D$ on $X$ the ``~{\it residue} of $Z$ w.r.t. $D$~" is the scheme $Res_D(Z)$
defined by the ideal sheaf $\Ii_Z:\Ii_D$, where $\Ii_Z,\Ii_D$ are the ideal sheaves
of $Z$ and $D$ respectively. The multiplication by local equations of $D$ defines
the exact sequence of sheaves:
\begin{equation}\label{resid}
0\to \Ii_{Res_D({ {Z}})}(-D)\to \Ii_{{Z}} \to
\Ii_{D\cap {{Z}},D}\to 0
\end{equation}
where the rightmost sheaf is the ideal sheaf of $D\cap Y$ in $D$.
\medskip

We will identify elements $T\in\PP^M$, which is the space of
embedding of $Y=\PP^{n_1}\times\dots\times\PP^{n_k}$, as tensors
of type  $(n_1+1)\times\dots\times (n_k+1)$ (modulo scalars).

\begin{definition}\label{Definition:NonRedundant}
A finite reduced subset $S\subset Y=\PP^{n_1}\times\dots\times\PP^{n_k}$
is a {\emph {decomposition}} of $T$ if $T\in\spa{\nu(S)}$ (with an abuse of notation sometime we will say that
also $\nu(S)\subset X$ is a decomposition of $T$). If moreover $T\notin\spa{\nu(S')}$ for any
$S'\subsetneq S$, the decomposition $S$ is said to be {\emph{ not-redundant}}.
Finally, if $\sharp S= \min \{\sharp S' \, | \, S'\subset Y \hbox{ and }T \in \langle \nu (S') \rangle \}$
then  $S$ is called a {\emph{minimal}} decomposition of $T$.
\end{definition}

\begin{remark} Clearly ``~not-redundant~" does not imply ``~minimal~".
As we will detail in Theorem \ref{add11}, it is alway possible to build a non-minimal non-redundant decomposition.
\end{remark}

Our target is to study the identifiability of a tensor $T\in\PP^M$
(i.e. when $T$ has only one minimal decomposition) by means of the
knowledge of the numbers $h^0(\Ii_A(\mathbf{u}))$, for all
$\mathbf{u}\subset\{1,\dots,k\}$, where $A{ \subset Y}$ is
a finite set which corresponds to a decomposition of $T$.

We will use the following notions for the rank of $T$.

\begin{definition}\label{Definition:Ranks}

The {\emph {rank}, $\mathrm{rk}(T)$, of $T$} is the minimum $r$
for which there exists a minimal decomposition of $T$.
%reduced subscheme  $Y\subset \PP^{n_1}\times\dots\times\PP^{n_k}$ with $\sharp Y=r$ and $T\in\spa{Y}$.

The {\emph {cactus rank}} of $T$ is the minimum $r$ for which there exists a zero-dimensional subscheme
$\Gamma\subset X$ with $\deg(\Gamma)=r$ and $T\in\spa{\Gamma}$.
\end{definition}

Clearly:
\begin{center} rank of $(T) \geq$  cactus rank of $(T)$.
\end{center}
%We will use the corresponding notion of cactus decomposition of $T$.
In analogy to the rank case, we will say that a zero-dimensional scheme $Z\subset Y$
is a ``~minimal cactus decomposition~" of $T\in \mathbb{P}^M$ if $Z$ is of minimal
degree among the zero-dimensional schemes $Z'\subset Y$ such that $T\in \langle \nu (Z') \rangle$.
\medskip

If $\sigma_r$ is the $r$-secant variety of $X$,
then all tensors of rank $r$ belong to $\sigma_r$.

For any tensor $T$ of rank $r$, let $\Ss(T)$ denote the set of all (reduced) finite subsets
$S\in Y$ of cardinality $r$ such that $T\in\spa{\nu(S)}$.
Of course for all $S\in\Ss(T)$ the image $\nu(S)$ is linearly independent,
for otherwise $T$ is contained in the span of a subset
of cardinality $r'<r$, thus it has rank smaller than $r$.

A tensor $T$ is {\em identifiable} when $\Ss(T)$ is a singleton. The abstract
product  $Y=\PP^{n_1}\times\dots\times\PP^{n_k}$ is ``~{\em generically identifiable} in rank $r$~"
if the general $T\in\sigma_r$ is identifiable.

The main tool for our analysis of the identifiability of a tensor
$T\in\PP^M$ relies in the following proposition that is an immediate consequence of \cite[Lemma 1]{bb}
if we consider the residue exact sequence of $A\cup B$ cut by a linear space containing $A$.

\begin{proposition}\label{bb}
Consider linearly independent zero-dimensional schemes $A,B\subset
Y$. Then the  linear spans of the images $\nu(A),\nu(B)$ in the
Segre map satisfy
$$ \dim(\spa{\nu(A)}\cap\spa{\nu(B)}) = \dim(\spa{\nu (A\cap B)}+h^1(\Ii_{A\cup B}(\mathbf{1})).$$
\end{proposition}

\section{Rank}\label{Section:Rank}

If we know a decomposition $T=T_1+\dots+T_r$ of $T$ in terms of tensors $T_i$
of rank $1$, in general we cannot directly conclude that $r$ is the rank of $T$.

Our analysis will prove that, for small values of $r$, the rank of
$T$ is $r$ provided that the summands correspond to points  in a
suitably general geometric position. We will give a criterion to
find a lower bound for the cactus rank (and therefore also for the
rank).

%We fix through this section a product
%$$X =  \PP^{n_1}\times\dots\times\PP^{n_k}$$
%so that the span $\PP^M$ of $\nu(X)$ can be identified (mod scalars)
%with the space of tensors of type $(n_1+1)\times\dots\times (n_k+1)$.

\begin{theorem}\label{add7}
Fix a partition $ E\sqcup F=\{1,\dots ,k\} $ of the $k$ factors of the abstract product
$Y=\PP^{n_1}\times\dots\times\PP^{n_k}$, i.e.  $E = \{a_1,\dots ,a_{k-h}\}$ and
$F =\{b_1,\dots ,b_h\}$ for some fixed $0<h<k$. Let $M_F:=\prod _{i=1}^{h} (n_{b_i}+1)$ be
the affine dimension of the ambient space of the Segre  embedding of the $F$-factors $Y_F:=\PP^{b_1}\times\dots\times \PP^{b_h}$.
Now fix $0<c<M_F$ and let
%integers $b>0$, $0<x<k$ and assume given a partition $\{1,\dots ,k\} = E\sqcup F$
%with $E = \{a_1,\dots ,a_{n-x}\}$ and $F =\{c_1,\dots ,c_x\}$.
%Set $m=\prod _{i=1}^{x} (n_{c_i}+1)$ and assume $b< m$.
%Let
$A\subset {{Y}}$ be a zero-dimensional scheme which
satisfies $h^1(\Ii _A(E)) =0$ and $h^0(\Ii _A(F)) < M_F -c$. Take
any $T\in \spa{\nu (A)}$ such that $T\notin\spa{\nu (A')}$ for any
$A'\subsetneq A$. Then there are no zero-dimensional schemes
$B\subset {{Y}}$ with $\deg(B)\le c$ such that $T\in
\spa{\nu (B)}$.
\end{theorem}

\begin{proof}
Notice that the Segre embedding of the projection $\pi_F$ maps
${{Y}}$ to $\PP^{{M_F}-1}$.

Since $h^1(\Ii _A(E)) =0$, we have $h^1(\Ii _A(1)) =0$. The condition $h^0(\Ii _A(F)) < M_F-c$
implies that $\deg (A)> c$. Assume that the theorem fails and take $B\subset Y$ with $\deg (B)\le c$
and $T\in \spa{\nu (B)}$.
Since $\deg (A)> \deg (B)$ and $T\notin\spa{\nu (A')}$ for any $A'\subsetneq A$, we have $B\nsubseteq A$.
Moreover $h^1(\Ii _{A\cup B}(1)) >0$, by \refprop{bb}.
More precisely, since $T\in \spa{\nu (B)}$ and $T\notin \spa{\nu (A')}$ for any $A'\subsetneq A$, we have
\begin{equation}\label{h1s}
h^1(\Ii _{A\cup B}(1)) > h^1(\Ii _{A'\cup B}(1))
\end{equation}
for all $A'\subset A$ with $A\cup B\neq A'\cup B$.

Since $\deg (B)\le c< M_F$, we have $h^0(\Ii _B(F)) > 0$. Take a
general divisor $D\in |\Ii _B(F)|$. In other words,  $D$ is the
inverse image in ${{Y}}$ of a hyperplane in the Segre
embedding of $Y_F$ and $B\subset D$. Since $h^0(\Ii _A(F)) < M_F
-c \le h^0(\Ii _B(F))$, then $A\nsubseteq D$, so that $(D\cap
A)\cup B$ is strictly contained in $A\cup B$. Hence by
\refeqn{h1s} we get $h^1(\Ii _{(D\cap A)\cup B}(1)) < h^1(\Ii
_{A\cup B}(1))$. The residual exact sequence  \refeqn{resid}
applied to $D$  gives $h^1(\Ii _{\Res _D(A\cup B)}(E))>0$. Since
$\Res _D(A\cup B)\subseteq A$,  we get a contradiction.
\end{proof}

Observe that the condition $h^1(\Ii _A(E)) =0$ can be satisfied only when $\deg(A) \le \prod _{i=1}^{k-h} (n_{a_i}+1)$,
the affine dimension of the ambient space of the Segre embedding of the $E$-factors $Y_E=\PP^{a_1}\times\dots\times \PP^{a_{k-h}}$.
\medskip

We can rephrase \refthm{add7} to produce results on the rank of tensors.

\begin{corollary}\label{add7cor} With the previous notation, if $T$ sits in the linear span of a scheme $A\subset {{Y}}$ which
satisfies the assumptions of \refthm{add7}, then the cactus rank
of $T$ is a least $c+1$. Hence also the rank of $T$ cannot be
smaller than $c+1$.
\end{corollary}

{\begin{corollary}\label{alessandra} Let $A\subset Y$ be a
zero dimensional scheme of $\deg(A)=c+1$ (resp. a finite set with
$\sharp(A)=c+1$). With the Notation of Theorem \ref{add7}, assume
that $h^1(\Ii_A(E))=0$ and $h^1(\Ii_A(F))=0$. Then any $T\in
\langle \nu(A)\rangle$ such that $T\notin \langle \nu(A')\rangle$
for any $A'\subsetneq A$ has cactus rank (resp. cactus rank and
rank) equal to $c+1$.
\end{corollary}
\begin{proof} It is straightforward from Theorem \ref{add7}
\end{proof}
}

Let we point out an application to the case of 3-way tensors.

\begin{proposition}\label{333}
Consider $k=3$ and let $T$ be a tensor of type
$(n_1+1)\times(n_2+1)\times (n_3+1)$, which has a not-redundant
 decomposition $T=T_1+\dots+T_r$, where the $T_i$'s are tensors of rank $1$. Identify each $T_i$ with a point
in $X=\nu(\PP^{n_1}\times\PP^{n_2}\times\PP^{n_3})$ and set $A=\{T_1,\dots,T_r\}\subset X$.
Call $A_E$ (resp. $A_F$) the projection of $A$ to $Y_E=\PP^{n_1}\times\PP^{n_2}$ (resp. $Y_F=\PP^{n_3}$).

Assume that $A$ has different coordinates, $A_E$ is linearly independent and $A_F$ is contained in no
hyperplanes of $\PP^{n_3}$. Then the rank of $T$ is at least $n_3+1$.
\end{proposition}
\begin{proof}
In the notation of \refthm{add7}, take $E=\{1,2\}$ and $F=\{3\}$. Then our assumptions on $A_E, A_F$
imply that $A$ satisfies  $h^1(\Ii _A(E)) =h^0(\Ii _A(F))=0$. Thus $T$ cannot have a decomposition
with $M_F-1=n_3$ summands.
\end{proof}

\begin{example}
Kruskal's Theorem for the identifiability of tensors (\cite{Kru}) provides results similar
to the previous proposition for the rank. We notice that the numerical range of application of
\refprop{333} is sometimes wider than Kruskal's range.

For instance, consider a tensor $T$
of type $3\times 4\times 6$ having a decomposition with $6$ summands. If the decomposition
determines a subset $A$ satisfying the assumptions of \refprop{333}, we can conclude
that the rank of $T$ is $6$. We cannot get the same conclusion directly with
Kruskal's Theorem because we are outside Kruskal's numerical range, since $6>(3+4+6-2)/2$.
\end{example}

The following example should be considered as a step towards the
Strassen's Conjecture on the rank of a block tensor (see \cite{Stra}).

\begin{example}\label{Strassen}
\refprop{333} can give results on the rank of a sum of tensors, when we have some information
on the decompositions of the summands.

For instance, consider again tensors of type $3\times 4\times 6$ and take a tensor $T$
which is the sum $T=T'+T''$ of two tensors of rank $3$. Consider a decomposition of $T'$ (resp. $T''$)
in a sum of three tensors of rank $1$ and call $S'$ (resp. $S''$) the set of cardinality $3$
in the product $\PP^2\times\PP^3\times\PP^5$ determined by the decomposition.

If the set $S=S'\cup S''$ has cardinality $6$ and satisfies the assumptions of \refprop{333}
(i.e. both $\pi_{\{1,2\}}(\nu(S))$ and $\pi_{\{3\}}(\nu(S))$ are linearly independent),
then we can conclude that the rank of $T$ is $6$.
\end{example}

We show below that, on the contrary, if we increase the
cardinality of a decomposition, we can always construct new
non-redundant decompositions of a tensor $T$.

\begin{example}\label{add9}
Fix $P\in {Y}$ and write $P =(p_1,\dots ,p_k)$ with
$p_i\in \PP^{n_i}$. Assume $n_i>0$. Take two points $b_i,c_i\in
\PP^{n_i}$ such that $p_i\neq b_i,c_i$ but $p_i$ is contained in
the line of $\PP^{n_i}$ spanned by $b_i$ and $c_i$. Let $O_i:=
(u_1,\dots ,u_k)$, $Q_i:= (v_1,\dots ,v_k)$ be the points of
${Y}$ with $u_j=v_j$
%{$=a_j$ questo $=a_j$ non dovrebbe esserci} 
for all $j\ne i$, $u_i=b_i$ and $v_i=c_i$. We
have $\nu ({P})\in \spa{\{\nu (O_i),\nu (Q_i)\}}$, and of course
$\nu ({P})\notin \spa{\nu(S')}$ for any $S'\subsetneq
\{O_i,Q_i\}$.
\end{example}

We show that indeed the previous construction  can yield a {\it
non-redundant} decomposition. Moreover the following also show
that having found a non-redundant decomposition does not imply
that it is a minimal one.

\begin{theorem}\label{add11}
Assume $n_i>0$ for at least one $i$.
Take a finite set $A\subset X$ of cardinality $r\leq M$, such that $\nu (A)\subset X$ is linearly independent.
 Take a general $T\in \spa{\nu(A)}$. Then there
exists a non-redundant decomposition $S\subset Y$ of $T$ of
cardinality $r+1$.
%finte subset $S\subset X$ such that $\sharp (S) =r+1$, $T\in \spa{\nu(S)}$ and
%$T\notin \spa{\nu(S')}$ for any $S'\subsetneq S$.
\end{theorem}

\begin{proof}
Fix $P=(p_1,\dots ,p_k)\in A$ and take $Q_1, O_1$ as in Example \ref{add9}, with $i=1$,
and with the additional condition that $O_1, Q_1\notin (A\setminus \{P\})$.
 We may take $O_1$ to be a general point of $\PP^{n_1}\times \{p_2\}\times \cdots \times \{p_k\}$.
Hence we may take $O_1 = (a_1,p_2,\dots ,p_k)$, with $a_1$ general.
Set $A':= A\setminus \{P\}$.

\quad (a) Assume $\nu (O_1)\notin \spa{\nu (A)}$. This is always possible unless
$\spa{\nu(A)}$ contains $\PP^{n_1}\times\{p_2\}\times\dots\times\{p_k\}$.
Since $P\in A$, this is equivalent to $\nu (Q_1)\notin \spa{ \nu (A)}$.
Set $S:= (A\setminus \{P\})\cup \{O_1,Q_1\}$.
We have $\sharp (S) =r+1$ and $\spa{ \nu (S)} \supseteq \spa{\nu (A)}$ so that $T\in \spa{\nu (S)}$.
The set $S$ does not depend on $Q_1$, but only on $A$ and $P$. To prove that
$S$ satisfies the claim,  it is sufficient to prove that for a general $T\in  \spa{\nu (A)}$ there is no $S'\subsetneq S$ with $T\in  \spa{\nu (S')}$.
It is sufficient to test all subsets of $S$ with cardinality $r$. Take $S'\subset S$ with $\sharp (S') =r$.
If $S'\supset A'$, i.e. if either $S' =A'\cup \{O_1\}$ or $S' = A'\cup \{Q_1\}$,
then $\spa{\nu (S')} \cap \spa{\nu (A)} = \spa{\nu (A')}$, because $\{\nu (O_1),\nu (Q_1)\}\cap \spa{\nu (A)} =\emptyset$. Thus
$T$ cannot stay in $\spa{\nu (S')}$ because it cannot stay in $\spa{\nu (A')}$  for a proper subset $A'$ of $A$.
Hence $S'\nsupseteq A'$. Since $\sharp (S') = r$, we have $S' = A''\cup \{O_1,Q_1\}$ with $A''\subset A'$ and $\sharp (A'') =r-2$.
If $T\in \spa{\nu (S')}$, then $\spa{\nu (S')} \cap \spa{\nu (A)} \supseteq \spa{\nu (A'')\cup \{T\}}$. Since $\nu (O_1)\notin  \spa{\nu (A)}$,
we get:
$$\spa{\nu (S')} \cap \spa{\nu (A)} = \spa{\nu (A'')\cup \{T\}}.$$
The left hand side of this equality does not depends on the choice of $T$. Varying $T\in \spa{\nu (A)} \setminus \spa{\nu (A')}$
we get $\spa{\nu (S')} \supseteq \spa{\nu (A)}$. Since $\sharp (S') =\sharp (A)$ and $\nu (A)$ is linearly independent,
we get  $\spa{\nu (S')} = \spa{\nu (A)}$ and hence $\nu (O_1)\in\spa{\nu (A)}$, a contradiction.

\quad (b) If $n_j=0$ for all $j>1$ we are done. Assume for
instance that  $n_2>0$ and  that $\spa{\nu(A)} \supset
\PP^{n_1}\times\{p_2\}\times\dots\times\{p_k\}$, so that we cannot
take $\nu (O_1)\notin \spa{\nu (A)}$. Since $\nu ({P}) \in \spa{
\{\nu (O_1),\nu (Q_1)\}}$, we get $\spa{\nu (\{O_1\}\cup A')}
=\spa{\nu (A)}$.  Replace thus $A$ with $A'\cup \{O_1\}$ and $P$
by $O_1$. Then take the construction of \refex{add9} with $i=2$.
If the new general point $O_2$ satisfies $\nu (O_2)\notin \spa{\nu
(A)}$, then we conclude as in step (a). Otherwise we have
$\spa{\nu(A)} \supset
\PP^{n_1}\times\PP^{n_2}\times\{p_3\}\times\dots\times\{p_k\}$
Then continue  with $i=3$ in the construction of Example
\ref{add9}, and so on. After at most $k$ steps we must conclude by
step (a), because $\spa{A}$ cannot contain the whole product
${Y}$, since $\sharp (A)\le M$.
\end{proof}

{
\begin{remark}\label{add11.01}
Take $P =(p_1,\dots ,p_k)$, $i\in \{1,\dots ,k\}$ with $n_i>0$ and
$b_i,c_i$ as in Example \ref{add9}. Notice that in the previous
proof we can choose for $c_i$ any point (different from $b_i$) in
the line spanned by $b_i,p_i$. Thus, in Theorem \ref{add11} we get
a positive-dimensional family of sets $S\subset Y$ such that
$\sharp (S) =r+1$, $T\in \spa{\nu(S)}$ and $T\notin \spa{\nu(S')}$
for any $S'\subsetneq S$.

%We fix $P$ (and hence $p_i$) and $b_i$, but not $c_i$. We take as
%$c_i$ any points $\notin \{b_i,c_i\}$, but contained in the line
%spanned by $b_i$ and $p_i$. We get a positive-dimensional family
%of examples and hence in Theorem \ref{add11} we get a
%positive-dimensional families of set $S\subset X$ such that
%$\sharp (S) =r+1$, $T\in \spa{\nu(S)}$ and $T\notin \spa{\nu(S')}$
%for any $S'\subsetneq S$.
\end{remark}}

We end this section with a discussion on some consequence of \refthm{add7} on symmetric tensors.
\smallskip

Assume $n:= n_1=\dots = n_k$ so that $Y$ is a product of $k$ copies of $\PP^n$.
The Segre map restricted to the diagonal $\Delta\subset Y$ can be identified with the {\it Veronese map}
$v_k: \PP^n\to \PP^D$ where $D+1=\binom{k+n}n$. The space $\PP^D$ 
parameterizes symmetric tensors $T$, which in turn can be identified with homogeneous polynomials (forms)
of degree $k$ in $n+1$ variables. The {\it symmetric rank} is the minimum $r$ for which there exists
a finite subset $A\subset\PP^n$ of cardinality $r$ with $T\in\spa{v_k(A)}$.

A conjecture raised in \cite{CGLM} and well known  as {\it Comon's Conjecture} predicts
that the symmetric rank of a symmetric tensor $T$ always coincides with the rank of $T$ as a normal tensor
in the span of $\nu(Y)$.

The following corollary of \refthm{add7} implies  that, if  some assumptions on a minimal symmetric decomposition $A$ of a symmetric tensor $T$ are satisfied,  that the symmetric rank of $T$ coincides with its rank (and cactus rank).

\begin{corollary}\label{comon}
With the previous notation, consider  a zero-dimensional scheme $A\subset\PP^n$ of degree $\deg(A)=c+1$.
Call $\Jj_A$ the ideal sheaf of $A$ in $\PP^n$ and assume that $h^1(\Jj_A(e))=0$ for some $e\leq k/2$.
Take $T\in \spa{v_k(A)}$ such that $T\notin \spa{v_k(A')}$ for any $A'\subsetneq A$.
 Then $T$ has cactus rank equal to $c+1$. If $A$ is reduced
 (i.e. it is a finite set of points), then $T$ has also rank $c+1$.
\end{corollary}
\begin{proof} Consider any subset $E\subset\{1,\dots,k\}$ of cardinality  $e$ and take $F=\{1,\dots,k\}\setminus E$.
Notice that $f:=\sharp F\geq e$, so that also $h^1(\Jj_A(f))=0$.
This implies that, considering $A$ as a subset of $\Delta\subset
Y$,  in the notation of Theorem \ref{add7}, we have $h^1(\Ii
_A(E)) =0$ and $h^0(\Ii _A(F)) < M_F -c$. The claim follows from
\refthm{add7}.
\end{proof}

 We show that \refcor{comon} provides new evidences for the Comon's Conjecture,
 sometimes even in a numerical range larger  than the ones considered in previous works on the topic
 (e.g. \cite{Fri} and \cite{ZHQ}).

\begin{remark}\label{numer}
Assume $k=2e$ even. Then by \refcor{comon}, the Comon's conjecture holds for symmetric tensors $T$ having
a minimal symmetric decomposition $A$ with $h^1(\Jj_A(e))=0$.

The condition $h^1(\Jj_A(e))=0$ holds for general subsets $A\subset\PP^n$, as soon as $c+1=\sharp A$ satisfies
\begin{equation}
                  c+1\leq r_0:=\binom{n+e}e.
\end{equation}
So Comon's Conjecture holds for general tensors whose symmetric rank $c+1$ is bounded by $r_0$.

When $k=2e+1$ is odd, a similar conclusion holds with $r_0:=1+ \binom{n+e}e$.
\end{remark}

\begin{example}
After the Alexander-Hirshowitz classification of defective Veronese varieties (\cite{AH}), a general symmetric tensor
in the span of $v_k(\PP^n)$ is known to have symmetric rank
$$ r_g=\lceil \frac{\binom {n+k}k}{n+1}\rceil $$
except for a list of few exceptional cases.

In general $r_g$ is bigger than our bound $r_0$ of \refrem{comon}, since, for fixed $e$, $r_g$ grows asymptotically as
$(2^ne^n)/(n+1)!$ while $r_0$ grows like $e^n/n!$.

Nevertheless, there are cases in which $r_g$ and $r_0$ coincide. This happens for $(k,n)=(6,2)$ or $(k,n)=(8,2)$.
Since these two cases are not in the list of exceptional Veronese varieties, we can conclude that Comon's Conjecture
holds for general forms of degree $6$ and degree $8$ in $3$ variables.
\end{example}

\begin{example} Take $(k,n)=(4,3)$. Then \refrem{numer} tells us that the Comon's conjecture holds for
general symmetric tensors in the span of $v_4(\PP^3)$ with a decomposition with $10$ summands. 
On the other hand, in this case the number $r_g$ is $9$, smaller than $r_0$.

Indeed  $(k,n)=(4,3)$ is in the list of exceptional cases in the
Alexander-Hirscho-witz theorem, so that a general form of degree
$4$ in $\PP^3$ has symmetric rank $10>r_g$. Thus Comon's
conjecture also holds for such general forms.

In some sense, \refthm{add7} provides a new heuristic reason why
the case $(k,n)=(4,3)$ is exceptional: a general tensor with a non-redundant decomposition with $10$ summands
cannot have a decomposition with $9$ summands.
A similar remark holds for the other exceptional cases of even degree: quadrics in any $\PP^n$ and
quartics in $\PP^2$ and $\PP^4$.
\end{example}

\section{Identifiability}\label{Section:Identifiability}

In order to get results on the identifiability of a tensor $T$, we
need to refine the previous analysis, and we are going to do that
in this section.

We will need the following terminology for the {\it Segre function}
of a finite subset of the product, introduced in \cite{CS16}.
\medskip

\begin{definition}
For any set of points $S\subset Y$, the {\em Segre function}
$\SF_S:\{1,\dots,k\}\to\NN$ is defined by:
\begin{center}
$\SF_S(i)=1+$ the dimension of the linear span of $\nu(\pi_{\{1,\dots,i\}}(S))$.
\end{center}
\end{definition}

Remark that the knowledge of the sequence
$h^0(\Ii _S (\{1,\dots ,i\}))$, $i=1,\dots ,k$, is equivalent to the knowledge
of the Segre function $\SF_S$.

More precisely, the definition of Segre function depends on the ordering of the factors of the product.
The knowledge of $h^0(\Ii_S(\mathbf{u}))$, for all possible $\mathbf{u}\subset\{1,\dots,k\}$, is equivalent to the knowledge
of the Segre functions of $S$ under all  possible re-arrangements of the factors.
\medskip

Let us recall the following definition for a minimal dependent set
of point (the same can be found in \cite[Definition 2.9]{CS16},
while {{where}} in \cite[Chap.7 Sec.1]{GKZ}  a
minimal dependent set of points is called a {\it circuit}.).

\begin{definition}
A {set} of points $S\subset\PP^m$ is {\emph {minimally
dependent}},
%or that $S$ is a {\it circuit},
if $S$ is linearly dependent, but any proper subset { of
$S$}
%o a $S$
is linearly independent.
\end{definition}

We need now a list of results on the cohomology of $\Ii_S$, for finite sets $S$ in a product of projective spaces.
%We fix again through this section a product
%$$X = \PP^{n_1}\times\dots\times\PP^{n_k}$$
%and call $\PP^M$ the span of the image of $X$ in the Segre map $\nu$.

\begin{definition}
We say that a finite subset $S\subset {Y}$ is {\em
degenerate} if there exists an index $i$ such that $\pi_i(S)$ does
not span $\PP^{n_i}$, i.e. there is a hyperplane $H\subset
\PP^{n_i}$ such that
$$S\subset  \PP^{n_1}\times\dots\times \PP^{n_{i-1}}\times H\times \PP^{n_{i+1}}\times\dots\times\PP^{n_k}.$$
We say that $T\in\PP^M$ is {\em degenerate} if there exists a
degenerate subset $A\subset \nu({Y})$ such that
$T\in\spa{A}$.
\end{definition}

Notice that if $S\subset {Y}$ is non-degenerate, then
necessarily $\sharp S>\max\{n_i\}$.
\smallskip

The two results below are mainly an extension to the non-symmetric case of results
in \cite{BGL} and \cite{bb}.

\begin{lemma}\label{ee2}
Fix a finite set $S\subset
%X=\nu (\mathbb{P}^{n_1}\times \cdots \times \mathbb{P}^{n_k})$
 {Y=\mathbb{P}^{n_1}\times \cdots \times
\mathbb{P}^{n_k}}$ such that $x:= \sharp (S)\le k+1$, $h^1(\Ii
_S(1))
>0$ and $h^1(\Ii _{S'}(1)) =0$ for each $S'\subsetneq S$. Then
there is $E\subset  \{1,\dots ,k\}$ such that $\sharp (E) =k+2-x$
and $\sharp (\pi _i(S)) =1$ for all $i\in E$.
\end{lemma}
\begin{proof}
The lemma is trivial if $x=2$, because $\Oo _Y(1)$ is very ample and so the assumptions
of the lemma are never satisfied in this case. Thus we may assume $x>2$ and use induction on the integer $x$.

The case $x=3$ is also true, because $\nu({Y})$ is cut out
by quadrics and each line contained in $\nu({Y})$ is the
image by $\nu$ of a line in one of the $k$ factors of
${Y}$.

Thus we assume $x>3$ (and so $k\ge 3$).
\\
Assume $\sharp (\pi _1(S))\ge 2$, so that the Segre function of $S$ satisfies $\SF_S(1) \ge 2$. Notice that
the assumptions on $S$ imply that $\SF_S(k)=x-1$.
Take points  $P=(a_1,\dots ,a_k)$, $Q=(b_1,\dots ,b_k)\in S$ with $P\ne Q$.
Since $h^1(\Ii _{S'}(1)) =0$ for all $S'\subsetneq S$ and $x>3$,
we may find $A\supset \{P,Q\}$ with $\sharp (A) =x-1$ and $h^1(\Ii _A(1))=0$.
Since $x\le k+1$, there is a minimal integer $i\in \{2,\dots ,k\}$ such that $\SF_S(i-1) =\SF_S(i)$.
By \cite[Proposition 2.5]{CS16} for every minimally dependent $S'\subseteq S$ with respect to the line
bundle $\Oo(\{1,\dots,i\})$ we have $\sharp (\pi _i(S')) =1$. Since $i\ge 2$ and $P\ne Q$, we have
$h^1(\Ii _{\{P,Q\}}(\{1,\dots,i\}))=0$. Hence we may find a
minimally dependent set containing $\{P,Q\}$. Thus $a_i=b_i$. Take any $C=(c_1,\dots ,c_k)\in
S\setminus \{P,Q\}$. If $c_1\ne a_1$ we may take minimally dependent $S''\subseteq S$ with respect to the line
bundle $\Oo(\{1,\dots,i-1\})$ containing $\{P,C\}$ and hence $c_i=a_i$.
If $c_1= a_1$ we may take minimally dependent $S''\subseteq S$ with respect to the line
bundle $\Oo(\{1,\dots,i-1\})$ containing $\{Q,C\}$ and hence $c_i=b_i =a_i$. Thus $\pi _i(S) =\{a_i\}$.
The lemma is now proved when $x=k+1$, by setting $E=\{i\}$ .

If $x<k+1$, we apply the same
proof to the projection $\pi_{\mathbf{u}}(X)$ where $\mathbf{u}=\{1,\dots,k\}\setminus \{i\}$
and conclude by descending induction on the integer $k+1-x$.
\end{proof}

\begin{lemma}\label{ee3} Fix a finite set $S\subset {Y}$ such that $x:= \sharp (S)\le k+n_1$,
$\langle \pi _1(S)\rangle=\PP^{n_1}$, $h^1(\Ii _S(1)) >0$ and
$h^1(\Ii _{S'}(1)) =0$ for each $S'\subsetneq S$. Then there is
$E\subset  \{2,\dots ,k\}$ such that $\sharp (E) =k+2-x$ and
$\sharp (\pi _i(S)) =1$ for all $i\in E$.
\end{lemma}

\begin{proof}
Take $n_1+1$ points $A_1,\dots ,A_{n_i+1}$, say $A_j = (a_{j,1},\dots ,a_{j,k})$, $1\le j \le n_1+1$,
with the property that the set $\{a_{1,1},\dots ,a_{n_1+1,1}\}$ spans $\PP^{n_1}$.
Then just repeat the proof of \reflem{ee2}.
\end{proof}

Now, with the aid of \refprop{bb}, we are ready to prove the
following theorem where we explicit  a sufficient condition for a
non-redundant decomposition to be minimal and unique.

\begin{theorem}\label{ee4} Fix an element $T\in \mathbb{P}^M$ and a non-redundant decomposition $S\subset Y$ of $T$ and let $\sharp(S)=r$.
If $m:= \max \{n_1,\dots ,n_k\}$,
% Fix a non-degenerate finite set $A\subset X$.
%Set $r:=\sharp (A)$. Fix $T\in \spa{\nu(A)}$ such that
%$T\notin \spa{\nu(A')}$ for any $A'\subsetneq A$. Then:
then
\begin{enumerate}[(a)]
\item\label{a} If $2r\le k+m$, then the rank of $T$ is $r$.

\item\label{b} If moreover $2r<k+m$, then $\Ss (T) =\{S\}$, i.e. $T$ is identifiable.
\end{enumerate}
\end{theorem}

\begin{proof}
By permuting the factors of ${Y}$ we may assume that
$n_1=m$. The assumptions on $T$ imply in particular that $\nu (S)$
is linearly independent, i.e. $h^1(\Ii _S(1)) =0$. Let $S'\subset
Y$ be a minimal decomposition of $T$ such that $\sharp S' \leq r$
and $S'\neq S$.
 %  finite reduced set
%$S'\ne S$ of cardinality such that $T\in \spa{\nu(S')}$ and $\sharp S'\leq r$.
%We may assume that $S'$ is minimal,
%so that $T\notin \spa{\nu(S'')}$ for any $S''\subsetneq S'$. Thus $S\nsubseteq S'$ and $S'\nsubseteq S$.
%We have
Then, by \refprop{bb}, $h^1(\Ii _{S\cup S'}(1)) >0$.

Let $\tilde S\subseteq S\cup S'$ be a minimal subset of $S\cup S'$ containing $S$
and with $h^1(\Ii _{\tilde S}(1)) >0$. Since $\tilde S\supseteq S$, and $X$ is non-degenerate,
the set $\pi _1(\tilde S)$ spans $\PP^{n_1}$.
By Lemma \ref{ee3} we have $\sharp (\tilde S)\ge k+n_i+1$ and hence $\sharp (S')+r \ge k+m+1$.

If $2r=k+m$ we get $\sharp (S')\ge r$. This proves that $T$ has rank $r$.
If $2r<k+m$ we get a contradiction.
\end{proof}

\begin{corollary}\label{ee5}
Set $m:= \max \{n_1,\dots ,n_k\}$. Take a non-degenerate $T\in \PP^M$.
If $2\mathrm{rk}(T)<k+m$,  %twice the rank of $T$ is smaller than $k+m$,
then $T$ is identifiable.
\end{corollary}

\begin{proof}
Take $S\in \Ss (T)$. Since $S\subset Y$ and $T\in \spa{\nu(S)}$,
then $X$ {chi e' X?} is the minimal multiprojective space
containing $S$. Then apply part (\ref{b}) of \refthm{ee4}.
\end{proof}

\begin{remark}\label{compare} One cannot give a result on the identifiability of tensors without comparing
it with the celebrated Kruskal's bound (\cite{Kru}), which is known to be sharp (\cite{sb}, \cite{Der}).

Our condition on the decomposition $S$ of the tensor $T$ is weaker than the condition
imposed by Kruskal, which requires to compute the span of any subset of $\pi_i(S)$
up to cardinality $n_i+1$ (in order to determine the Kruskal's rank), while we only need
to check that $\pi_i(S)$ generates $\PP^{n_i}$.
Since the Kruskal's rank of the projections of $S$ in principle can be even $1$ (when $\pi_i$ is
not injective), for low values of the rank our result determines the identifiability of $T$
under weaker assumptions.

Of course, as our assumptions are weaker than Kruskal's ones, we cannot give
applications outside  Kruskal's numerical range. There are few cases in which
the numerical range of application of our result matches with Kruskal's range.
One of them e.g. is given by tensors of type $3\times 2 \times 2 \times 2$.
\end{remark}

Next, we show that under some condition on the decomposition $S$
of a tensor $T$, we can prove that any other decomposition $S'$ of cardinality
smaller or equal than $\sharp S$ must have projections in special position.

\begin{definition}
A zero-dimensional scheme $Z\subset {Y}$ is said to be
\emph{curvilinear} if each connected component of $Z$ has
embedding dimension $1$, i.e. there exists a smooth curve in
${Y}$ containing $Z$.
\end{definition}

We will indeed prove the result even when $S'$ is non-reduced, provided
that $S'$ is {\it curvilinear}.

Notice that, of course, any {\it reduced} finite subset of $Y$ is
curvilinear.

\begin{theorem}\label{add6}
Set $m':= \min \{n_1,\dots ,n_k\}$. Fix integers $s>0$, $0<x<k$
such that $(m'+1)^{k-x} \ge r$. Let $B\subset Y$ be
zero-dimensional curvilinear scheme  and $S\subset Y$ be a finite
set with different coordinates such that $\sharp (S)=r$, $\deg (B)
=x$  and $h^1(\Ii _S (\mathbf{u})) =0$ for any subset
$\mathbf{u}\subset\{1,\dots ,k\}$ of cardinality $k-x$.

Assume that each projection $\pi_i$ is an isomorphism when restricted to $B$
(when $B$ is reduced this is equivalent to say that also $B$ has different coordinates).

Then,  $h^1(\Ii _{S\cup B}(1)) =0$.
\end{theorem}
\begin{proof} Set $Z:= S\cup B$ and assume $h^1(\Ii _Z(1)) >0$.
Taking $S\setminus (B\cap S)$ instead of $S$ we reduce to the case
$S\cap B =\emptyset$. Fix $\mathbf{u}=\{2,\dots,k\}$ so that
$\pi_{\mathbf{u}}$ is the projection to the last $k-1$ factors and
write $Y_{\mathbf{u}}=\PP^{n_2}\times\dots\times\PP^{n_k}$
%$X_{\mathbf{u}}=\PP^{n_2}\times\dots\times\PP^{n_k}$.

\quad (a) First assume $k=2$ and hence $x=1$, ${\mathbf{u}}=\{2\}$ and $\sharp (S)\le m+1$.
Write $B = \{O\}$ with $O =(O_1,O_2)$. Take a general hyperplane $H\subset \PP^{n_1}$ such that
$O_1\in H_1$ and set $W_1:= H_1\times \PP^{n_2}$. Since $H_1$ is a general hyperplane containing $O_1$
and $S$ has different coordinates, either $W_1\cap S =\emptyset$ or $W_1\cap S$
is the unique point whose image by $\pi _1$ is $O_1$. Hence $\sharp (W_1\cap Z)\le 2$.

Since $\Oo(1)$ is very ample, then $h^1(\Ii _{W_1\cap Z}(1)) =0$,
so $h^1(W_1,\Ii _{W_1\cap Z,W_1}(1)) =0$. The residual exact
sequence of $W_1$ in ${Y}$ gives $h^1(\Ii _{Z\setminus
(Z\cap W_1)}({\mathbf{u}}))>0$. Since $Z\setminus (Z\cap
W_1)\subseteq S$, we get $h^1(\Ii _S({\mathbf{u}})) >0$, a
contradiction.

\quad (b) Now assume $k\ge 3$ and assume that the claim holds for multiprojective spaces with $k-1$ factors.
Write $\{P(1),\dots ,P(y)\}$ for the points of the reduced set $B_{\red}$, where $1\le y \le x$.
For any $i\in \{1,\dots ,k\}$ let $\Aa _i$ be
the set of all pairs $(P({\mathbf{u}}),Q({\mathbf{u}}))\in B_{\red}\times S$ such that all the coordinates of $P({\mathbf{u}})$ and $Q({\mathbf{u}})$
are the same, except the $i$-th one (which is different, because $S\cap B=\emptyset$).
Assume the existence of $(P({\mathbf{u}}),Q({\mathbf{v}}))\in \Aa _i$ and $(P({\mathbf{u}}),Q({\mathbf{v}}))\in \Aa _j$.
Since $S\cap B=\emptyset$, $Q({\mathbf{u}})$ (resp. $Q({\mathbf{v}})$) has all coordinates equal to the one of $P({\mathbf{u}})$,
except the $i$-th (resp. $j$-th) one, which is different.
Since $k\ge 3$ and $S$ has different coordinates, we get $Q({\mathbf{u}}) =Q({\mathbf{v}})$.
Hence $i=j$. Since $y\le x<k$, there is $h\in \{1,\dots ,k\}$ with $\Aa _h=\emptyset$.
Permuting the factors of $Y$ we may assume $h=1$, i.e. $\Aa _1 =\emptyset$.
This is equivalent to the injectivity of the map $\pi _{1|S\cup B_{\red}}$.
Since $\pi _{1|B}$ is an embedding, we get that $\pi_{1|Z}$ is an embedding.

Fix $P\in B_{\red}$ and call $P_1,\dots ,P_k$ its components. Take
a general hyperplane $H_1$ of $\PP^{n_1}$ containing $P_1$. Since
$\pi _1(B)$ is curvilinear and $H_1$ is general, we have
$\pi_1(B)\cap H_1 =\{P_1\}$ (scheme-theoretic intersection). Set
$W_1:= H_1\times {Y}_{\mathbf{u}}$. $W_1$ is an element of
$|\Oo(\{1\})|$. Since $\pi_{1|B}$ is an embedding and
$\pi_1(B)\cap H_1 =\{P_1\}$ (scheme-theoretic intersection), we
have $B\cap W_1=\{P\}$ (scheme-theoretic intersection). Since
$H_1$ is general and $S$ has different coordinates, either
$W_1\cap S = \emptyset$ or $W_1\cap S$ is the unique point of $S$
with $P_1$ as its first coordinate. Hence $Z_1:= Z\cap W_1$ is
always reduced, it contains $P$ and at most another point, which
is in $S$. Set $Z_{\mathbf{u}}:= \Res _{W_1}(Z)$, $B_1:= \Res
_{W_1}(B)$ and $S_1:= \Res _{W_1}(S) =S\setminus (S\cap W_1)$. We
have $Z_{\mathbf{u}} =B_1\cup S_1$, $\deg (B_1)=x-1$, $B_1\subset
B$, $S_1\subseteq S$. Since $\sharp (Z_1)\le 2$ and $\Oo(1)$ is
very ample, we have $h^1(\Ii _{Z_1}(1)) =0$. Hence $h^1(W_1,\Ii
_{Z_1}(1)) =0$. The residual exact sequence of $W_1$ in $X$ gives
$h^1(\Ii _{Z_{\mathbf{u}}}({\mathbf{u}})) >0$. Since
$\pi_{\mathbf{u}}{1|Z}$, is an embedding, then
$\pi_{{\mathbf{u}}|Z_{\mathbf{u}}}$ is an embedding. Hence
$h^1(\Ii _{Z_{\mathbf{u}}}({\mathbf{u}}))=
h^1({Y}_{\mathbf{u}},\Ii _{Z_{\mathbf{u}}}(1))$. The
inductive assumption on the number of factors of the
multiprojective space gives $h^1({Y}_{\mathbf{u}},\Ii
_{Z_{\mathbf{u}}}(1))=0$, a contradiction.
\end{proof}

Mixing the previous theorem with \refprop{bb}, we get the following.

\begin{corollary} Set $m':= \min \{n_1,\dots ,n_k\}$ and
fix integers $s>0$, $0<x<k$ such that $(m'+1)^{k-x} \ge r$. Take a
tensor $T\in \PP^M$ with a non-redundant decomposition $S\subset
Y$ such that $\sharp S=r$ and $S$ has different coordinates.
Assume $h^1(\Ii _S ({\mathbf{u}})) =0$ for any subset
${\mathbf{u}}\subset\{1,\dots ,k\}$ of cardinality $k-x$.

Then any other non-redundant decomposition $S'\subset X$ of $T$ of
cardinality $\leq x$ cannot have different coordinates.
\end{corollary}

The statement of \refthm{add7} cannot produce corollaries on the identifiability, because the assumptions
do not include the case $\sharp A=\sharp B$.

We can however modify the proof of \refthm{add7}, adding some assumptions on the decomposition $S$,
which produces results which bound different the decompositions of a tensor $T$.

\begin{definition}\label{ee6}
Let $W\subset \PP^r$ be an integral and non-degenerate projective variety. A finite set (resp. zero-dimensional scheme)
$S\subset W$ is said to be \emph{set-theoretically quasi-general} (resp.  \emph{scheme-theoretically
quasi-general}) if the set-theoretic (resp. scheme-theoretic) intersection $W\cap \spa{S}$ is set-theoretically
(resp. scheme-theoretical) equal to $S$.
\end{definition}

\begin{lemma}\label{a1}
Let $W\subset \PP^M$ be an integral and non-degenerate variety of
dimension $n$. Fix a general reduced subset $S\subset  W$ with
$\sharp S \le M-n-1$. Then $S$ is the scheme-theoretic base locus
of the linear system on $Y$ { chi e' Y?} induced by
$H^0(\Ii _S(1))$, i.e $S$ is scheme-theoretically quasi-general.
\end{lemma}

\begin{proof}
Let $L\subset \PP^M$ be a general linear space of dimension $M-r-1$. By Bertini's theorem the scheme $L\cap W$ is a finite set
of $\deg (Y)$ points, moreover the set $L\cap W$ is in linearly general position in $L$. Hence for any $S\subset L\cap W$ with
$\sharp S \le M-r-1$, the restriction
of $U:= H^0(\Ii _S(1))$ to $L\cap W$ has $S$ as its set-theoretic base locus. Since $L$ is a linear space, the restriction of $U$ to
$W$ has base locus contained in $L\cap W$. Thus $S$ is the base locus of the restriction of $U$ to $W$.
Since $W$ is integral and non-degenerate, a general subset $A\subset W$ with cardinality at most
$M-n$ spans a general subspace of $\PP^M$ with dimension $\sharp (A)-1$. The claim follows.
\end{proof}

\begin{proposition}\label{add14}
Fix a partition $ E\sqcup F=\{1,\dots ,k\} $ with
$E = \{a_1,\dots ,a_{n-h}\}$ and $F =\{b_1,\dots ,b_h\}$ for some $0<h<k$, and a positive integer $c< M_F=\prod _{i=1}^{h} (n_{b_i}+1)$.
%integers $r>0$, $0<h<k$ and assume given a partition $\{1,\dots ,k\} = E\sqcup F$ with
%$E = \{a_1,\dots ,a_{n-h}\}$ and $F =\{b_1,\dots ,b_h\}$
%and $r< \prod _{i=1}^{h} (n_{b_i}+1)$.
Let $Z\subset Y$ be a zero-dimensional scheme such
that $\sharp (Z)=c$, $h^1(\Ii _Z(E)) =0$ and $h^1(\Ii _Z (F)) =0$. Assume that $\pi _F(Z)$ is
set-theoretically (resp. scheme-theoretically) quasi-general. Take any $T\in \spa{\nu(Z)}$ such
that $T\notin \spa{\nu(Z')}$ for any $Z'\subsetneq Z$.

If $S$ is a  finite set (resp. zero-dimensional scheme) such that $S\ne Z$, $\deg (S)\le b$ and
$T\in \spa{\nu(S)}$, then $\deg (S)=b$ and $\pi _F(S) =\pi _F(Z)$.
\end{proposition}

\begin{proof}
By Theorem \ref{add7} it is sufficient to do the case $\deg (S)=c$
and $h^0(\Ii _S(F)) =h^0(\Oo _Y(F)) -b$. In particular $\pi _F$
induces an embedding of $S$ into
$Y_F==\PP^{b_1}\times\dots\times\PP^{b_k}$.

The proof of Theorem \ref{add7} works verbatim, if there is $H\in |\Oo _Y(F)|$ containing $S$, but not containing $Z$.
Since $h^0(\Ii _S(F)) =h^0(\Ii _Z(F))$, this is  equivalent to require that $H^0(\Ii _S(F)) \ne H^0(\Ii _Z(F))$,
i.e. that $\pi _F(S)$ is not contained in the base locus of $\pi _F(Z)$. Since $\pi _F(Z)$
is the set-theoretic (resp. scheme-theoretic) quasi-general, the base locus
of $|\Oo _{Y_F}(1)|$ is $\pi _F(Z)$. Thus we get $\pi _F(S) =\pi _F(Z)$.
\end{proof}

\begin{corollary}\label{add14cor}
For each $i\in \{1,\dots ,k\}$ fix a set $F_i\subset \{1,\dots ,k\}$ such that $i\in F_i$ and set
$E_i:= \{1,\dots ,k\}\setminus F_i$. Let {$S\subset Y$} be a finite set such that $\sharp (S) =r$, where:
$$r <\prod _{j\in F_i} (n_j+1)\quad\mbox{ and }\quad r\le \prod _{h\in E_i} (n_h+1).$$
Take any $T\in \spa{\nu(S)}$, such that $T\notin \spa{\nu(S')}$ for any $S'\subsetneq S$. Assume that:
$$h^0(\Ii _S(F_i)) =h^0(\Oo _Y(F_i))-r \quad\mbox{ and }\quad h^0(\Ii _S(E_i)) =h^0(\Oo _Y(E_i))-r$$
for all $i=\{1,\dots ,k\}$. Assume moreover that each $\pi _{F_i}(S)$ is set-theoretically
(resp. scheme-theoretical) quasi-general.

If $S'\ne S$ is a zero-dimensional subscheme of degree $\le r$ such that
 $\spa{\nu (S')}$ contains $T$,  then $S'$ is a finite set,
$\sharp (S')=r$ and $\pi _i(S') =\pi _i(S)$ for all $i$.
\end{corollary}

Notice that (unfortunately) we are not able to conclude that $S=S'$: they can be different
even if  $\pi _i(S') =\pi _i(S)$ for all $i$. Namely the points can differ by a rearrangement of the
coordinates.

\begin{proof}
Since $h^0(\Ii _S(F_i)) =h^0(\Oo _Y(F_i))-\sharp (S)$, each $\pi _{F_i|S}$ is injective.
Assume that $S'$ exists. By Proposition \ref{add14} applied to $F_i$, $S'$ is a finite set with
$\sharp (S')=r$, $\pi _{F_i|S'}$ injective and $\pi _{F_i}(S')=\pi _{F_i}(S)$. Thus
$\pi _i(S') =\pi _i(S)$ for all $i$.
\end{proof}

\begin{remark} {\refcor{add14cor} does not provide the identifiability of a tensor $T$: 
it simply bounds strictly the locus where different decompositions
of the same tensor $T$ could lie.  We observe that, on} the other hand, the numerical range of application of
 \refcor{add14cor} is wider than {the range of Kruskal's criterion of identifiability}.

 Just to give an example, consider tensors of type $3\times 3\times 6$,  corresponding (mod scalars) to points
 in the space $\PP^{53}$ which contains the Segre embedding of $\PP^2\times\PP^2\times\PP^5$.
 Kruskal's criterion for the identifiability applies only when the rank $r$ is bounded by $r\leq (3+3+6-2)/2=5$.
 Our \refcor{add14cor} applies,  taking $F_1=F_2=\{1,2\}$, $F_3=\{3\}$ and
 checking the geometric assumptions, even for $r=6$.
\end{remark}


\begin{thebibliography}{CGLM08}

\bibitem[AH95]{AH} Alexander~J. and A. Hirschowitz~A.,
\newblock{\em Polynomial interpolation in several variables},
\newblock{J. Algebraic Geom.}, 4 (1995), 201–222.

\bibitem[BB12]{bb} Ballico~E.  and Bernardi~A.,
\newblock {\em Decomposition of homogeneous polynomials with low rank},
\newblock {Math. Z.} 271 (2012), 1141--1149.

\bibitem[BGI11]{bgi} Bernardi~A., Gimigliano~A. and Id\`{a}~M.,
\newblock{\em Computing symmetric rank for symmetric tensors},
\newblock{ J. Symbolic. Comput.} { {46}} (2011), 34--55.

\bibitem[BGL13]{BGL} Buczy\' nski~J., Ginenski~A. and Landsberg~J.~M.,
\newblock{\em Determinantal equations for secant varieties and
the Eisenbud-Koh-Stillman conjecture},
\newblock{J. London Math. Soc.} 88 (2013), 1--24.

\bibitem[COV17]{COV3} Chiantini~L., Ottaviani G.  and Vannieuwenhoven~N.,
\newblock{\em Effective criteria for specific identifiability of tensors and forms},
\newblock{SIAM J. Matrix Anal. Applic. to appear},
\newblock{(arXiv:1609.00123)}.

\bibitem[CS16]{CS16} Chiantini~L. and Sacchi~D.,
\newblock {\em Segre functions in Multiprojective Spaces and Tensor Analysis},
\newblock {in: G. Casnati et al., \em From  Classical to Modern Algebraic Geometry,
Corrado Segre's Mastership and Legacy},
Springer - Trends in History and Science (2016), 361--374.

\bibitem[CS11]{cs} Comas~G, M. Seiguer~M.,
\newblock{\emph{On the rank of a binary form}},
\newblock{Found Comput Math} 11 (2011),  65--78.

\bibitem[CGLM08]{CGLM} Comon~P., Golub~G., Lim~L.~H., and Mourrain~B.,
\newblock{\em Symmetric tensors and symmetric tensor rank},
\newblock{SIAM J. Matrix Anal. Applic.}, 30 (2008), 1254--1279.


\bibitem[D13]{Der} Dersken~H.,
\newblock {\em Kruskal's uniqueness inequality is sharp},
\newblock {Lin. Alg. Applic.},
438 (2013), 708--712.

\bibitem[F16]{Fri} Friedland~S.,
\newblock{\em Remarks on the symmetric rank of symmetric tensors},
\newblock{SIAM J. Matrix Anal. Applic.}, 37 (2016), 320--337.

\bibitem[GKZ]{GKZ} Gelfand~I.~M., M.~Kapranov~M. and Zelevinsky~A,
\newblock{\bf Discriminants, Resultants, and Multidimensional Determinants.}
\newblock{ Birk\"auser, Zurich} (1994).

\bibitem[K77]{Kru}
Kruskal~J.~B.,
\newblock {\em Three-way arrays: rank and uniqueness of trilinear decompositions,
with applications to arithmetic complexity and statistics},
\newblock {Lin. Alg. Applic.} 18 (1977), 95--138.

\bibitem[SB00]{sb}
Sidiropoulos~N.~D. and Bro~R.,
\newblock {\em On the uniqueness of multilinear decomposition of N-way arrays},
\newblock {J. Chemometrics} 14 (2000), 229--239.

\bibitem[S73]{Stra} Strassen~V.,
\newblock {\em Vermeidung von Divisionen},
\newblock {J. Reine Angew. Math.} 264 (1973), 184-–202.


\bibitem[ZHQ16]{ZHQ} Zhang~X., Huang~Z.~H., and Qi~L.,
\newblock{\em Comon's conjecture, rank decomposition, and symmetric rank
decomposition of symmetric tensors},
\newblock{SIAM J. Matrix Anal. Applic.}, 37 (2016), 1719--1728.

\end{thebibliography}
\end{document}